\documentclass[12pt]{article}

\usepackage{hyperref}
\usepackage{graphics}
\usepackage{graphicx}
\usepackage{amssymb, amsmath, amsfonts, amsthm, epsfig}
\usepackage{epsfig,latexsym}
\usepackage{multicol}
\usepackage{pgf,tikz}
\usetikzlibrary{arrows}
\begin{document}
\sloppy

\newfont{\blb}{msbm10 scaled\magstep1}

\makeatletter
\newtheorem*{rep@theorem}{\rep@title}
\newcommand{\newreptheorem}[2]{%
\newenvironment{rep#1}[1]{%
 \def\rep@title{#2 \ref{##1}}%
 \begin{rep@theorem}}%
 {\end{rep@theorem}}}
\makeatother

\makeatletter
\newtheorem*{rep@prob}{\rep@title}
\newcommand{\newrepprob}[2]{%
\newenvironment{rep#1}[1]{%
 \def\rep@title{#2 \ref{##1}}%
 \begin{rep@prob}}%
 {\end{rep@prob}}}
\makeatother

\makeatletter
\newtheorem*{rep@conj}{\rep@title}
\newcommand{\newrepconj}[2]{%
\newenvironment{rep#1}[1]{%
 \def\rep@title{#2 \ref{##1}}%
 \begin{rep@conj}}%
 {\end{rep@conj}}}
\makeatother

\makeatletter
\newtheorem*{rep@cor}{\rep@title}
\newcommand{\newrepcor}[2]{%
\newenvironment{rep#1}[1]{%
 \def\rep@title{#2 \ref{##1}}%
 \begin{rep@cor}}%
 {\end{rep@cor}}}
\makeatother

\makeatletter
\newtheorem*{rep@prop}{\rep@title}
\newcommand{\newrepprop}[2]{%
\newenvironment{rep#1}[1]{%
 \def\rep@title{#2 \ref{##1}}%
 \begin{rep@prop}}%
 {\end{rep@prop}}}
\makeatother

\theoremstyle{plain}
\newtheorem{thm}{Theorem}[section]
\newtheorem{ques}[thm]{Question}
\newtheorem{lem}[thm]{Lemma}
\newtheorem{prop}[thm]{Proposition}
\newtheorem{cor}[thm]{Corollary}
\newtheorem{conj}[thm]{Conjecture}
\newtheorem{defn}[thm]{Definition}
\newtheorem{prob}[thm]{Problem}
\newtheorem*{rem}{Remark}
\newtheorem*{note}{Note}
\newreptheorem{theorem}{Theorem}
\newrepprob{prob}{Problem}
\newrepconj{conj}{Conjecture}
\newrepconj{cor}{Corollary}
\newrepprop{prop}{Proposition}

\newtheorem*{123}{1-2-3 Conjecture}
\newtheorem*{12}{1-2 Conjecture}
\newtheorem*{list123}{List 1-2-3 Conjecture}
\newtheorem*{list12}{List 1-2 Conjecture}
\newtheorem*{ACconj}{Additive Colouring Conjecture}
\newtheorem*{AClistconj}{Additive List Colouring Conjecture}
\newtheorem*{weakAClistconj}{Weak Additive List Colouring Conjecture}
\newtheorem*{TCC}{Total Colouring Conjecture}
\newtheorem*{LCC}{List Colouring Conjecture}
\newtheorem*{LLL}{Lov\'asz Local Lemma}
\newtheorem*{SLL}{Symmetric Local Lemma}
\newtheorem*{MLL}{Modified Local Lemma}
\newtheorem*{CN}{Combinatorial Nullstellensatz}
\newtheorem*{RyserFormula}{Ryser's Formula}

\newcommand{\lmulti}{\{\hspace{-0.035in}\{}
\newcommand{\rmulti}{\}\hspace{-0.035in}\}}
\newcommand{\per}{\textup{per}\,}
\newcommand{\mind}{\textup{mind}}
\newcommand{\tmind}{\textup{tmind}}
\newcommand{\pind}{\textup{pind}}
\newcommand{\ch}{\textup{ch}}
\newcommand{\M}{\mathbb{M}}
\newcommand{\N}{\mathbb{N}}
\newcommand{\Z}{\mathbb{Z}}
\newcommand{\R}{\mathbb{R}}
\newcommand{\Q}{\mathbb{Q}}
\newcommand{\C}{\mathbb{C}}
\newcommand{\F}{\mathbb{F}}
\newcommand{\pr}{\mathbf{P}}
\newcommand{\Ls}{\mathcal{L}}
\newcommand{\PP}{\mathcal{P}}
\newcommand{\D}{\Delta}
\newcommand{\tw}{\textup{\rm tw}}

\newcommand{\Se}{\chi_\Sigma^e}
\newcommand{\Pe}{\chi_\Pi^e}
\newcommand{\me}{\chi_m^e}
\newcommand{\se}{\chi_s^e}
\newcommand{\Sv}{\chi_\Sigma^v}
\newcommand{\Pv}{\chi_\Pi^v}
\newcommand{\mv}{\chi_m^v}
\newcommand{\sv}{\chi_s^v}
\newcommand{\St}{\chi_\Sigma^t}
\newcommand{\Pt}{\chi_\Pi^t}
\newcommand{\mt}{\chi_m^t}
\newcommand{\st}{\chi_s^t}
\newcommand{\chSe}{\ch_\Sigma^e}
\newcommand{\chPe}{\ch_ \Pi ^e}
\newcommand{\chme}{\ch_m^e}
\newcommand{\chse}{\ch_s^e}
\newcommand{\chSv}{\ch_\Sigma^v}
\newcommand{\chPv}{\ch_ \Pi ^v}
\newcommand{\chmv}{\ch_m^v}
\newcommand{\chsv}{\ch_s^v}
\newcommand{\chSt}{\ch_\Sigma^t}
\newcommand{\chPt}{\ch_ \Pi ^t}
\newcommand{\chmt}{\ch_m^t}
\newcommand{\chst}{\ch_s^t}

\newcommand{\eSe}{{\chi'}_\Sigma^e}
\newcommand{\ePe}{{\chi'}_\Pi^e}
\newcommand{\eme}{{\chi'}_m^e}
\newcommand{\ese}{{\chi'}_s^e}
\newcommand{\eSv}{{\chi'}_\Sigma^v}
\newcommand{\ePv}{{\chi'}_\Pi^v}
\newcommand{\emv}{{\chi'}_m^v}
\newcommand{\esv}{{\chi'}_s^v}
\newcommand{\eSt}{{\chi'}_\Sigma^t}
\newcommand{\ePt}{{\chi'}_\Pi^t}
\newcommand{\emt}{{\chi'}_m^t}
\newcommand{\est}{{\chi'}_s^t}
\newcommand{\echSe}{{\ch'}_\Sigma^e}
\newcommand{\echPe}{{\ch'}_ \Pi ^e}
\newcommand{\echme}{{\ch'}_m^e}
\newcommand{\echse}{{\ch'}_s^e}
\newcommand{\echSv}{{\ch'}_\Sigma^v}
\newcommand{\echPv}{{\ch'}_ \Pi ^v}
\newcommand{\echmv}{{\ch'}_m^v}
\newcommand{\echsv}{{\ch'}_s^v}
\newcommand{\echSt}{{\ch'}_\Sigma^t}
\newcommand{\echPt}{{\ch'}_ \Pi ^t}
\newcommand{\echmt}{{\ch'}_m^t}
\newcommand{\echst}{{\ch'}_s^t}

\newcommand{\tSe}{{\chi''}_\Sigma^e}
\newcommand{\tPe}{{\chi''}_\Pi^e}
\newcommand{\tme}{{\chi''}_m^e}
\newcommand{\tse}{{\chi''}_s^e}
\newcommand{\tSv}{{\chi''}_\Sigma^v}
\newcommand{\tPv}{{\chi''}_\Pi^v}
\newcommand{\tmv}{{\chi''}_m^v}
\newcommand{\tsv}{{\chi''}_s^v}
\newcommand{\tSt}{{\chi''}_\Sigma^t}
\newcommand{\tPt}{{\chi''}_\Pi^t}
\newcommand{\tmt}{{\chi''}_m^t}
\newcommand{\tst}{{\chi''}_s^t}
\newcommand{\tchSe}{{\ch''}_\Sigma^e}
\newcommand{\tchPe}{{\ch''}_\Pi ^e}
\newcommand{\tchme}{{\ch''}_m^e}
\newcommand{\tchse}{{\ch''}_s^e}
\newcommand{\tchSv}{{\ch''}_\Sigma^v}
\newcommand{\tchPv}{{\ch''}_\Pi ^v}
\newcommand{\tchmv}{{\ch''}_m^v}
\newcommand{\tchsv}{{\ch''}_s^v}
\newcommand{\tchSt}{{\ch''}_\Sigma^t}
\newcommand{\tchPt}{{\ch''}_\Pi ^t}
\newcommand{\tchmt}{{\ch''}_m^t}
\newcommand{\tchst}{{\ch''}_s^t}

\newcommand{\we}{\chi_{\sigma^*}^e}
\newcommand{\We}{\chi_\sigma^e}
\newcommand{\wt}{\chi_{\sigma^*}^t}
\newcommand{\Wt}{\chi_\sigma^t}
\newcommand{\chwe}{\ch_{\sigma^*}^e}
\newcommand{\chWe}{\ch_\sigma^e}
\newcommand{\chwt}{\ch_{\sigma^*}^t}
\newcommand{\chWt}{\ch_\sigma^t}

\newcommand{\Ie}{\chi_{\cap}^e}
\newcommand{\It}{\chi_{\cap}^t}
\newcommand{\chIe}{\chi_{\cap}^e}
\newcommand{\chIt}{\chi_{\cap}^t}

\newcommand{\KLT}{\textrm{Karo{\'n}ski-{\L}uczak-Thomason}}

\newcommand{\spc}{\hspace{0.08in}}

\renewcommand{\thepage}{\arabic{page}}


\title{The 1-2-3 Conjecture and related problems: a survey}

\author{Ben Seamone}

\maketitle

\begin{abstract}
The 1-2-3 Conjecture, posed in 2004 by Karo\'nski, {\L}uczak, and Thomason, states that one may weight the edges of any connected graph on at least 3 vertices from the set $\{1,2,3\}$ (call the weight function $w$) so that the function $f(v) = \sum_{u \in N(v)} w(uv)$ is a proper vertex colouring.  This paper presents the current state of research on the 1-2-3 Conjecture and the many variants that have been proposed in its short but active history.
\end{abstract}

\section{Introduction}\label{ch:intro:KLTintro}

Unless otherwise stated, a graph $G = (V,E)$ is simple, finite, and undirected.  Standard graph theory notation (\cite{BM08}, \cite{Diestel}) is used throughout.

In 2004, Karo{\'n}ski, {\L}uczak, and Thomason \cite{KLT04} made the following conjecture:

\begin{123}[Karo{\'n}ski, {\L}uczak, Thomason \cite{KLT04}]
If $G$ is a graph with no component isomorphic to $K_2$, then the edges of $G$ may be assigned weights from the set $\{1,2,3\}$ so that, for any adjacent vertices $u,v \in V(G)$, the sum of weights of edges incident to $u$ differs from the sum of weights of edges incident to $v$.
\end{123}

The motivation for the 1-2-3 Conjecture comes from the study of graph irregularity strength.  An edge weighting of a graph $G$ is an {\bf irregular assignment} if, for any pair of vertices $u, v \in V(G)$, the sum of weights of edges incident to $u$ differs from the sum of weights of edges incident to $v$.  The {\bf irregularity strength} of a graph $G$ is the smallest value of $k$ such that $G$ has an irregular assignment from $[k]$.

With a few simple definitions, we may state the 1-2-3 Conjecture more succinctly.  An {\bf edge $k$-weighting} is a function $w:E(G) \to [k] := \{1,2, \ldots, k\}$.  An edge $k$-weighting $w$ is a {\bf proper vertex colouring by sums} if $\sum_{e \ni u}w(e) \neq \sum_{e \ni v}w(e)$ for every $uv \in E(G)$.  Denote by $\Se(G)$ the smallest value of $k$ such that a graph $G$ has a edge $k$-weighting which is a proper vertex colouring by sums.  
A graph $G$ is {\bf nice} if no connected component is isomorphic to $K_2$.  The 1-2-3 Conjecture may now be stated as follows:

\begin{123}[Karo{\'n}ski, {\L}uczak, Thomason \cite{KLT04}]
If $G$ is nice, then $\Se(G) \leq 3$.
\end{123}

This paper surveys what is currently known about the 1-2-3 Conjecture, as well as the interesting questions (and answers) that have arisen during the conjecture's brief but active history.
Since the introduction of the 1-2-3 Conjecture, a variety of variations have been considered.  For instance, one may vary the set to be weighted or coloured (edges, vertices, or both), the operation by which one obtains a colour, or the set from which one weights the objects (a variation of interest is to study weightings from arbitrary lists rather than from a fixed set).  Due to the number of different weighting-colouring parameters one could conceivably study (and since there is no standard notation that is consistent across the literature on the subject), we have introduced the notation above, which is a modification of notation proposed by Gy{\H o}ri and Palmer \cite{GP09}.  Note the three components of $\Se(G)$ -- ``$\chi$'' indicates that we desire a proper {\em vertex} colouring, the superscript ``$e$'' indicates that the weighting of interest is of the {\em edges}, and the subscript ``$\Sigma$'' indicates that vertex colours are obtained by {\em summing} edge weights.  As new parameters are defined for different combinations of weightings and colourings, we will develop notation consistent with this format.

\section{Progress on the 1-2-3 Conjecture}

Early approaches to the 1-2-3 Conjecture focused on relating $\Se(G)$ to $\chi(G)$.  One of the first such results appears in the paper which introduces the 1-2-3 Conjecture:

\begin{thm}[Karo{\'n}ski, {\L}uczak, Thomason \cite{KLT04}]\label{odd}
If $(\Gamma,+)$ is a finite abelian group of odd order and $G$ is a nice $|\Gamma|$-colourable graph, then there is a weighting of the edges of $G$ with the elements of $\Gamma$ such that the vertex colouring by sums is a proper vertex colouring.
\end{thm}

In particular, if $G$ is nice and $k$-colourable for $k$ odd, then $\Se(G) \leq k$.  This theorem was then extended to the following:

\begin{thm}
If $G$ is 2-connected and $\chi(G) \geq 3$, then $\Se(G) \leq \chi(G)$.  In particular, for any integer $k \geq 3$ and nice graph $G$, the following hold:
	\begin{enumerate}
	\item \textup{(Karo{\'n}ski, {\L}uczak, Thomason \cite{KLT04})} if $G$ is $k$-colourable for $k$ odd, then $\Se(G) \leq k$;
	\item \textup{(Duan, Lu, Yu \cite{DLY})} if $G$ is $k$-colourable for $k \equiv 0 \pmod 4$, then $\Se(G) \leq k$;
	\item \textup{(Duan, Lu, Yu \cite{DLY})} if $\delta(G) \leq k-2$, then $\Se(G) \leq k$;
	\item \textup{(Lu, Yang, Yu \cite{LYY})} if $G$ is $2$-connected, $k$-colourable, and has $\delta(G) \geq k+1$ for $k \equiv 2 \pmod 4$, then $\Se(G) \leq k$.
	\end{enumerate}
\end{thm}

It is also shown in \cite{DLY} that $\Se(G) \leq \chi(G)$ for a connected graph $G$ if $|V(G)|$ is odd or if there is a proper $\chi(G)$-colouring where one colour class has even size.  

The most significant progress toward solving the 1-2-3 Conjecture is the establishment and improvement of constant bounds on $\Se(G)$ for every nice graph $G$.  The best known bound to date is as follows:

\begin{thm}[Kalkowski, Karo{\'n}ski, Pfender \cite{KKP1}]\label{sum5}
If $G$ is a nice graph, then $\Se(G) \leq 5$.
\end{thm}

The proof of this bound involves an algorithmic argument; the graph's vertices are processed in a linear order and, at each step, some weights of edges incident to a vertex are adjusted so that there are no colouring conflicts with previously considered vertices.  
Addario-Berry, Dalal, McDiarmid, Reed, and Thomason \cite{ADMRT07} had previously shown that $\Se(G) \leq 30$ for any nice graph $G$, a bound which was then improved to $\Se(G) \leq 16$ by Addario-Berry, Dalal, and Reed \cite{ADR08}, then to $\Se(G) \leq 13$ by Wang and Yu \cite{WY08}, and then to $\Se(G) \leq 6$ by Kalkowski, Karo{\'n}ski, and Pfender \cite{KKP2}.

It is easily seen that there exist nice graphs for which two edge weights do not suffice to colour the vertices by sums (e.g. $K_3$, $C_6$), and hence the best possible constant bound for all nice graphs is the conjectured value of 3.  However, it is known that if $G$ is a random graph chosen from $\mathcal{G}_{n,p}$ for any constant $p \in (0, 1)$, 
then asymptotically almost surely $\Se(G) \leq 2$ \cite{ADR08}.  It is also known that $2$ edge weights suffice to properly colour vertices by sums for any digraph (in fact, a stronger result is known, which is stated in Theorem \ref{BGNdigraph}).

Clearly, $\Se(G) = 1$ if and only if adjacent vertices of $G$ always have different degrees.  
A classification of graphs for which $\Se(G) = 2$ does not yet exist, though some partial results are known.  Chang, Lu, Wu, and Yu \cite{CLWY} showed that $\Se(G) \leq 2$ if $G$ is bipartite and $d$-regular for $d \geq 3$.  Lu, Yu, and Zhang \cite{LYZ} proved that if $G$ is a nice graph which is either $3$-connected and bipartite or has minimum degree $\delta(G) \geq 8\chi(G)$, then $\Se(G) \leq 2$.
Davoodi and Omooni \cite{DO} claim to have recently proven that, for any two bipartite graphs $G$ and $H$, one has that $\Se(G \,\Box\, H) \leq 2$ if $G \,\Box\, H \neq K_2$.  Their manuscript also states that, for any two graphs $G$ and $H$, $\Se(G \,\Box\, H) \leq \max\{\Se(G), \Se(H)\}$ (though not stated in their paper, one should assume that both $G$ and $H$ are nice, or adopt the convention that $\Se(K_2) := \infty$).
Khatirinejad et al \cite{Ben1} show that $\Se(G) \leq 2$ if all cycles of $G$ have length divisible by $4$.  A {\bf generalized theta graph}, denoted $\Theta{_{(m_1,\ldots m_d)}}$ ($d \geq 3$) is a graph constructed from $d$ internally 
disjoint paths between distinct vertices, where the $i^{\rm th}$ path has of length $m_i$.
Khatirinejad et al \cite{Ben1} and Lu, Yang, and Zhang \cite{LYZ} show that $\Se(\Theta{_{(m_1,\ldots m_d)}}) = 2$ if and only if $\{m_1,m_2, m_3, \ldots m_d)\} \neq \{1, 4k_2 + 1, 4k_3+1, \ldots, 4k_d + 1\}$ for some set of positive integers $\{k_i \geq 1 \mid i=2,\ldots,d\}$.  
In fact, the {\em only} bipartite graphs for which it is known that $\Se(G) > 2$ are cycles $C_{4k+2}$ ($k \geq 1$), generalized theta graphs $\Theta{_{(1, 4k_2 + 1,4k_3+1 \ldots, 4k_d + 1)}}$, and an infinite subfamily of a class of bipartite graphs called polygon trees (Davoodi and Omoomi \cite{DO}).  The following problem remains open:

\begin{prob}
Characterize the set of graphs $\{G \mid \Se(G) \leq 2, \textrm{$G$ bipartite}\}$.
\end{prob}

Recalling Theorem \ref{odd}, one may consider the minimum $s$ such that, for any abelian group $(\Gamma,+)$ of order $s$, a graph $G$ has an edge weighting from $\Gamma$ which properly colours $V(G)$ by sums.  This parameter, called the {\bf group sum chromatic number} and denoted $\chi^{\Sigma}_g(G)$, was introduced by Anholcer and Cichacz \cite{AC1}, who prove the following theorem:

\begin{thm}[Anholcer, Cichacz \cite{AC1}]\label{groupsum}
If $G$ is a graph with no component having fewer than 3 vertices, then $\chi(G) \leq \chi^{\Sigma}_g(G) \leq \chi(G)+2$.
\end{thm}

Furthermore, the authors give a complete characterization of which graphs have group sum chromatic number $\chi(G)$, $\chi(G)+1$, and $\chi(G) + 2$.

We conclude this introduction to the 1-2-3 Conjecture with two related open problems.

\begin{prob}[Khatirinejad, Naserasr, Newman, Seamone, Stevens \cite{Ben1}]\label{unique3}
Does there exist a graph $G$ which has a unique edge $3$-weighting (up to isomorphism) which properly colours $V(G)$ by sums?
\end{prob}

We say that $G$ is {\bf $S$-weight colourable} if $G$ has an edge weighting from $S$ which properly colours $V(G)$ by sums, and $G$ is {\bf $k$-weight colourable} if $G$ is $S$-weight colourable for every set $S \subset \R$ of order $k$.
In \cite{Ben1}, gadget graphs are constructed which are uniquely $S$-weight colourable for any set $S$ of order $2$.  A positive answer to Question \ref{unique3} may similarly provide a class of gadget graphs useful for disproving the 1-2-3 Conjecture (and, hence, a negative answer would provide further evidence {\em for} the conjecture).

In \cite{Ben1}, it was also noted that it is unknown how difficult it is to decide if a given graph $G$ admits an edge $2$-weighting which properly colours $V(G)$ by sums (and, more generally, an edge $\{a,b\}$-weighting for general $a,b \in \R$).  

\begin{prob}[Khatirinejad, Naserasr, Newman, Seamone, Stevens \cite{Ben1}]
Is it NP-complete to decide whether or not a given graph $G$ is $2$-weight colourable?
\end{prob}

Havet, Paramaguru, and Sampathkumar \cite{HPS} have shown that it is NP-complete to decide if two edge weights suffice for a cubic graph $G$ if one requires distinct {\em multisets} of weights at adjacent vertices (more on such colouring variations can be found in the next section).  Dudek and Wajc \cite{DW} have recently proven that it is NP-complete to determine whether or not a graph $G$ is $S$-weight colourable for $S = \{1, 2\}$ and $S = \{0, 1\}$.  The general problem remains open, though they suggest that their methods may work for any set of two rational numbers.

\section{Variation I:  Colouring by products, multisets, sets, and sequences}\label{ch:intro:var1}

As mentioned in the introduction, there are a variety of ways in which one may modify the way in which vertex colours are obtained from an edge weighting of a graph.  The first variations we consider are those where addition of edge weights as the colouring method is replaced by another operation; in particular, we consider variations where colours are obtained by taking a product, multiset, set, or sequence of weights from edges incident to $v$ for each $v \in V(G)$.  
If such a colouring is proper, then the edge $k$-weighting of $G$ is a {\bf proper vertex colouring by products, multisets, sets, or sequences}, respectively.  Note that a colouring by sequences relies on a well-defined method for ordering the weights from edges incident to a vertex.  The natural way to do this is to first order the elements of $E(G)$, and then order the edges incident to a vertex according to the global ordering of $E(G)$. The smallest $k$ for which such colourings exist for a graph $G$ are denoted $\Pe(G)$ (products), $\me(G)$ (multisets), $\se(G)$ (sets), $\we(G)$ (sequences, where one can choose any ordering of $E(G)$), and $\We(G)$ (sequences, where a weighting must exist for any ordering of $E(G)$).  

Let us first consider colouring by multisets.  Clearly, two multisets are distinct if the sums of their elements are distinct, and hence colouring by multisets is a natural relaxation of the problem of colouring by sums.  The following simple observation is used frequently:

\begin{prop}\label{multibound}
If $\min\{\Se(G),\Pe(G),\se(G)\} \leq k$, then $\me(G) \leq k$.
\end{prop}

It follows from the 1-2-3 Conjecture and Proposition \ref{multibound} that the expected upper bound for $\me(G)$ is 3; whether or not this upper bound holds remains open.  The first bound on $\me(G)$ was established in \cite{KLT04}, where Karo{\'n}ski, {\L}uczak, and Thomason used a probabilistic argument to show that if $G$ is a nice graph, then $\me(G) \leq 183$.  By using the following vertex partitioning lemma, an improved bound of $\me(G) \leq 4$ was proven in \cite{AADR05}:

\begin{lem}[Addario-Berry, Aldred, Dalal, Reed \cite{AADR05}]\label{vpart}
Let $G$ be a connected graph which is not $3$-colourable.  There exists a partition of $V(G)$ into sets $V_0, V_1, V_2$ such that there exists a weighting $w:E(G) \rightarrow \{c_0, c_1, c_2, c^*\}$ with the following properties
	\begin{enumerate}
	\item for each $i \in \{0,1,2\}$, every $v \in V_i$ is incident to at least one edge weighted $c_i$;
	\item for each $i \in \{0,1,2\}$, the vertices in $V_i$ are incident only to edges weighted $c_i, c_{i-1 \pmod 3}, c^*$, and
	\item for each $i \in \{0,1,2\}$, if $u, v \in V_i$ are adjacent, then the number of edges incident to $u$ with weight $c_i$ is different from the number of edges incident to $v$ with weight $c_i$.
	\end{enumerate}
\end{lem}

\begin{thm}[Addario-Berry, Aldred, Dalal, Reed \cite{AADR05}]\label{me4}
If $G$ is a nice graph, then $\me(G) \leq 4$.
\end{thm}

\begin{proof}
If $G$ is $3$-colourable, then $\me(G) \leq \Se(G) \leq 3$ by Theorem \ref{odd} and Proposition \ref{multibound}.  If $G$ is not $3$-colourable, then the edge-weighting guaranteed by Lemma \ref{vpart} gives the desired vertex colouring by multisets.
\end{proof}

Addario-Berry et al. \cite{AADR05} also use Lemma \ref{vpart} to prove that $\me(G) \leq 3$ if $G$ is nice and $\Delta(G) \geq 1000$.  It is also shown in \cite{HPS} that $\me(G) \leq 2$ if $G$ is cubic and bipartite.

Colouring vertices by sequences of edge weights is a natural further relaxation of the problem of colouring by multisets.  Given an ordering of $E(G)$ and a weighting of $E(G)$, let each vertex be coloured by the sequence obtained by taking the multiset of weights from incident edges and ordering these weights according to the order in which their corresponding edges appear.  Recall that $\we(G)$ (respectively, $\We(G)$) is the smallest $k$ so that, for some (resp., any) ordering of $E(G)$, an edge $k$-weighting exists which properly colours $V(G)$ by sequences in this way.  Clearly, $\we(G) \leq \We(G) \leq \me(G)$, and so $4$ edge weights suffice for both colouring by sequences variations.  The following are the best known bounds for $\we(G)$ and $\We(G)$:

\begin{thm}[Seamone, Stevens \cite{Ben-LLL}]\label{sequencebounds}
If $G$ is a nice graph, then
\begin{enumerate}
\item $\we(G) \leq 2$, and
\item $\We(G) \leq 3$ if $\delta(G) \in \Omega\left(\log{\Delta(G)}\right)$.
\end{enumerate}
\end{thm}

In fact, stronger statements than those of Theorem \ref{sequencebounds} are proven in \cite{Ben-LLL}; see Theorem \ref{listsequencebounds}.  The constant suppressed by the $\Omega$ notation in the bound on $\delta(G)$ above can be improved if one imposes a girth condition on $G$.  Similar results for multigraphs are also proven in \cite{Ben-LLL}.

Turning our attention to colouring by products, we note that a constant bound on $\me(G)$ implies a constant bound on $\Pe(G)$.  The following corollary to the work of Addario-Berry et al. was first noted in \cite{SK08}:

\begin{cor}
If $G$ is a nice graph, then $\Pe(G) \leq 5$.
\end{cor}

\begin{proof}
If $G$ is $3$-colourable, then $\me(G) \leq 3$ by Theorem \ref{odd} and Proposition \ref{multibound}, and edge weights $\{2,3,5\}$ suffice.  If $G$ is not $3$-colourable, then let $w:E(G) \to \{c_0, c_1, c_2, c^*\}$ be the edge weighting guaranteed by Theorem \ref{vpart}, and let $c_0 = 2$, $c_1 = 3$, $c_2 = 5$ and $c^* = 1$; this edge $5$-weighting properly colours $V(G)$ by products.
\end{proof}

While constant bounds exist for $\Se(G)$, $\Pe(G)$ and $\me(G)$, this is not possible for the parameter $\se(G)$.  To see this, note that an edge $k$-weighting allows at most $2^k - 1$ possible vertex colours by sets and so the complete graph on $2^k$ vertices must have $\se(G) > k$. In the study of colouring by sets, Gy{\H o}ri and Palmer establish a link between $\se(G)$ and a hypergraph induced by $G$.  A hypergraph is said to have {\bf Property B} if there exists a colouring $c:V(H) \rightarrow \{1,2\}$ such that every hyperedge contains vertices of both colours; this concept has its roots in set theory and is due to Bernstein \cite{Bern}.

\begin{thm}[Gy{\H o}ri, Palmer \cite{GP09}]
Let $G$ be a bipartite graph with bipartition of the vertices $V(G) = X \cup Y$.  If $H$ is the hypergraph with $V(H) = Y$ and $E(H) = \{N_G(x) \,:\, x \in X\}$, then $\se(G) = 2$ if and only if $H$ has Property B.
\end{thm}

In \cite{GHPW08}, Gy{\H o}ri, Hor\u{n}\'ak, Palmer, and Wo\'zniak show that if $G$ is nice, then $\se(G) \leq 2\lceil{\log_2{\chi(G)}}\rceil + 1$, a bound which was subsequently refined.

\begin{thm}[Gy{\H o}ri, Palmer \cite{GP09}]\label{setbound}
If $G$ is nice, then $\se(G)  = \lceil{\log_2{\chi(G)}}\rceil + 1$.
\end{thm}

The parameter $\se(G)$ is denoted $\textup{gndi}(G)$ in \cite{GHPW08} and called the general neighbourhood distinguishing index.  The motivation for the parameter comes from the study of the {\bf neighbourhood distinguishing index}\label{ndidef} of a graph $G$, denoted $\textup{ndi}(G)$, which is the smallest integer $k$ such that $G$ has a proper edge $k$-colouring such that adjacent vertices have distinct sets of colours on their incident edges.  This parameter is also known as the {\bf adjacent vertex distinguishing chromatic index}, and the type of colouring is called an {\bf adjacent strong edge-colouring} or {\bf 1-strong edge-colouring}.  The neighbourhood distinguishing index was first introduced by Liu, Wang and Zhang \cite{LWZ}, who propose the following conjecture:

\begin{conj}[Liu, Wang, Zhang \cite{LWZ}]\label{ndi}
If $G \notin \{K_2, C_5\}$ and $G$ is connected, then $\Delta(G) \leq \textup{ndi}(G) \leq \Delta(G) + 2$.
\end{conj}

While the lower bound is trivially true, the upper bound appears to be difficult to prove.  Here are a few known bounds for ${\rm ndi}(G)$:

\begin{thm}[Balister, Gy{\H o}ri, Lehel, Schelp \cite{BGLS}]
Conjecture \ref{ndi} holds if $G$ is bipartite or $\Delta(G) \leq 3$.
\end{thm}

\begin{thm}[Balister, Gy{\H o}ri, Lehel, Schelp \cite{BGLS}]
If a graph $G$ is nice, then $\textup{ndi}(G) \leq \Delta(G) + O(\log{\chi(G)})$.
\end{thm}

\begin{thm}[Hatami \cite{H05}]
If a graph $G$ is nice and $\Delta(G) > 10^{20}$, then $\textup{ndi}(G) \leq \Delta(G) + 300$.
\end{thm}

Dong and Wang \cite{DW12} have recently introduced the study of {\bf neighbour sum distinguishing colourings}, where a proper edge colouring is used to properly colour vertices by sums.  The smallest $k$ for which such an edge-colouring exists is denoted $\textup{ndi}_\Sigma(G)$; it is shown that $\textup{ndi}_\Sigma(G) \leq \max\{2\Delta(G) + 1, 25\}$ if $G$ is a planar graph and $\textup{ndi}_\Sigma(G) \leq \max\{2\Delta(G), 19\}$ if $G$ is a graph such that $\textup{mad}(G) \leq 5$ (where $\textup{mad}(G)$ is the maximum average degree taken over all subgraphs of $G$).

Finally, we present one further variation of neighbour-distinguishing edge weightings, introduced by Baril and Togni \cite{BT}.  A {\bf proper $k$-tuple edge-colouring} of a graph $G$ assigns a set of $k$ colours to each edge such that adjacent edges have disjoint colour sets.  The vertex version of such a colouring was introduced by Stahl \cite{Stahl}.  Let $S(v)$ denote the union of all sets of colours assigned to edges incident to $v \in V(G)$.  If $S(u) \neq S(v)$ for each $uv \in E(G)$, then the edge colouring is a {\bf $k$-tuple neighbour-distinguishing colouring}.  Baril and Togni determine the smallest $k$ for which various classes of graphs have $k$-tuple neighbour-distinguishing colourings.  They also make the following conjecture, which extends Conjecture \ref{ndi} to multigraphs:

\begin{conj}[Baril, Togni \cite{BT}]\label{multindi}
If $G$ is a connected multigraph, $G \neq C_5$, with edge multiplicity $\mu(G)$ and maximum degree $\Delta(G)$, then ${\rm ndi}(G) \leq \Delta(G) + \mu(G) + 1$.
\end{conj}

\section{Variation II:  Total and vertex weightings}\label{ch:intro:var2}

\subsection{Total weightings}

A {\bf total weighting} of a graph $G$ is an assignment of a real number weight to each $e \in E(G)$ and each $v \in V(G)$.  The concept of proper colourings induced by total $k$-weightings was introduced by Przyby{\l}o and Wozniak \cite{PW10}.  Given a total $k$-weighting of $G$, we consider vertex colourings obtained by taking either the sum, product, multiset, or set of weights taken from the edges incident to $v$ and from $v$ itself for each $v \in V(G)$.  If such a colouring is proper, then the total $k$-weighting of $G$ is a {\bf proper vertex colouring by sums, products, multisets or sets}, respectively.  The smallest values of $k$ for which a proper colouring of each type exists for a graph $G$ are denoted $\St(G)$, $\Pt(G)$, $\mt(G)$ and $\st(G)$, respectively.  Note that, while one must exclude graphs with connected components containing precisely one edge when considering edge weightings which properly colour vertices by sums, this exclusion is not necessary for total weightings.

We first observe that, for the colouring methods considered so far, any total weighting parameter is bounded above by its edge weighting counterpart:

\begin{prop}\label{totalboundbyedge}
If $G$ is any graph, then
	\begin{enumerate}
	\item \textup{(Przyby{\l}o, Wozniak \cite{PW10})} $\St(G) \leq \Se(G)$,
	\item \textup{(Skowronek-Kazi\'ow, \cite{SK08})} $\Pt(G) \leq \Pe(G)$,
	\item $\mt(G) \leq \me(G)$,
	\item $\st(G) \leq \se(G) =\lceil{\log_2{\chi(G)}}\rceil + 1$.
	\end{enumerate}
\end{prop}

The proofs, while not difficult, are omitted.
In many cases, allowing vertex weights in addition to edge weights can strictly decrease the number of weights needed to obtain a proper vertex colouring of a graph.   The following conjecture motivates the study of total weightings and vertex colouring by sums:

\begin{12}[Przyby{\l}o, Wozniak \cite{PW10}] 
For every graph $G$, $\St(G) \leq 2$.
\end{12}

Hulgan, Lehel, Ozeki, and Yoshimoto \cite{HLOY} studied a variation of $\St(G)$ similar to that of the $S$-weight colourability variation of $\Se(G)$.  They consider weighting vertices and edges from a set of two non-negative real numbers $\{a,b\}$ and colouring vertices by sums; for some general values of $a,b \in \R$, they show that a variety of classes of graphs admit total weightings from $\{a,b\}$ which colour vertices by sums.

More generally, one may consider weightings which assign edges values from one set and vertices values from another.  For example, a {\bf total $(k,l)$-weighting} $w$ if a graph $G$ assigns $w(v) \in [k]$ for each $v \in V(G)$ and $w(e) \in [l]$ for each $e \in E(G)$.  The best known bound on $\St(G)$, to date, comes from the following result on total $(k,l)$-weightings:

\begin{thm}[Kalkowski \cite{Kal}]\label{12bound}
Every graph has a total $(2,3)$-weighting which is a proper vertex colouring by sums.  As a consequence, $\St(G) \leq 3$ for any graph $G$. 
\end{thm}

This theorem improves previous results from \cite{PW10}, where it is shown that $\St(G) \leq \min\{ \lfloor \chi(G)/2 \rfloor + 1, 11\}$ for any graph $G$, and from \cite{Prz08}, where it is shown that $\St(G) \leq 7$ if $G$ is regular.  
It is shown in \cite{PW10} that the 1-2 Conjecture holds if $G$ is complete, $3$-colourable, or $4$-regular.  

Little else is known when one varies the colouring operation.  Theorem \ref{12bound} implies that $\mt(G) \leq 3$ for any graph $G$.  The following conjecture is proposed on colouring by products from total $k$-weightings:

\begin{conj}[Skowronek-Kazi\'ow \cite{SK08}]\label{SKprod}
For every graph $G$, $\Pt(G) \leq 2$.
\end{conj}

In \cite{SK08}, Conjecture \ref{SKprod} was verified for graphs which are 3-colourable or complete, and it was shown that $\Pt(G) \leq 3$ for every graph $G$.



On the topic of total weightings which properly colour vertices by sets, much more attention has been paid to weightings which are themselves proper total colourings.  Recall that, in a proper total colouring of $G$, adjacent vertices, adjacent edges, and incident vertices and edges must all receive different colours.  As an extension of the neighbourhood distinguishing index discussed in Section \ref{ch:intro:var1} (see page \pageref{ndidef}), Zhang et al \cite{CLL+} study proper total colourings which, considered as weightings, also properly colour $V(G)$ by sets (the set of weights consisting of the colour of the vertex itself as well as the colours on all incident edges).  The fewest number of colours required for such a colouring to exist for a graph $G$ is called the {\bf adjacent vertex distinguishing total chromatic number}, which is denoted $\chi'_{at}(G)$.  Since $\chi'_{at}(G) \geq \chi''(G)$, it follows immediately that $\chi'_{at}(G) \geq \Delta(G)+1$ for any graph $G$.  Conversely, it is conjectured that any graph can be so coloured using a constant number of colours greater than its maximum degree.

\begin{conj}[Zhang, Chen, Li, Yao, Lu, Wang \cite{CLL+}]\label{adjstrong}
For any graph $G$, $\Delta(G) + 1 \leq \chi'_{at}(G) \leq \Delta(G) + 3$.
\end{conj}

This conjecture is reminiscent of Conjecture \ref{ndi} on edge colourings and, more notably, the Total Colouring Conjecture.

\begin{TCC}[Behzad \cite{Beh} and Vizing \cite{Viz68}]
For any graph $G$, $\chi''(G) \leq \Delta(G) + 2$.
\end{TCC}

The best known bound on a graph's total chromatic number is due to Molloy and Reed \cite{MR}, who show that $\chi''(G) \leq \Delta(G) + 10^{26}$.  The Total Colouring Conjecture is notoriously difficult, which suggests that Conjecture \ref{adjstrong} may also be very difficult to solve.  However, Coker and Johannson \cite{CJ} have recently shown that there exists a universal constant $C$ such that $\chi'_{at}(G) \leq \Delta(G) + C$.  Their proof relies on a probabilistic argument to show that $\chi'_{at}(G) \leq \chi''(G) + C'$ for some other constant $C'$, then invokes Molloy and Reed's aforementioned bound on $\chi''(G)$.

Pil\'sniak and Wo\'zniak \cite{PW+} study a similar parameter, where one requires that a proper total colouring of $G$ distinguishes adjacent vertices by sums.  They too conjecture that $\Delta(G) + 3$ colours should suffice for any graph $G$; they verify their conjecture for complete graphs, cycles, bipartite graphs, cubic graphs and graphs with maximum degree at most three.

\subsection{Vertex weightings}\label{vertexweightings}

Having considered edge and total weightings, we now turn our attention to vertex weightings.  A {\bf vertex $k$-weighting} of a graph $G$ is a mapping $w:V(G) \to [k]$; we consider the usual four methods of obtaining a vertex colouring from a vertex $k$-weighting, taking either the sum, product, multiset, or set of weights from vertices adjacent to $v$ for each $v \in V(G)$.  If such a colouring is proper, then the vertex $k$-weighting of $G$ is a {\bf proper vertex colouring by sums, products, multisets or sets}, respectively.  The smallest values of $k$ such that a colouring of each type exists for a graph $G$ are denoted $\Sv(G)$, $\Pv(G)$, $\mv(G)$ and $\sv(G)$, respectively.

A vertex weighting of $G$ which is a proper vertex colouring by sums is also known as an {\bf additive colouring} of $G$ (initially called a {\bf lucky labelling} of $G$), studied in \cite{CGZ} by Czerwi\'nski, Grytczuk, and \.Zelazny.  The value $\Sv(G)$ is called the {\bf additive colouring number} of $G$ (denoted $\eta(G)$ in \cite{CGZ}).  The following conjecture is proposed:

\begin{ACconj}[Czerwi\'nski, Grytczuk, \.Zelazny \cite{CGZ}]\label{lucky}
For every graph $G$, $\Sv(G) \leq \chi(G)$.
\end{ACconj}

It is shown in \cite{ADKM} that it is NP-complete to determine if a graph $G$ has $\Sv(G) = k$ for any $k \geq 2$, and that it is still NP-complete if one considers the subproblem of determining if a $3$-colourable planar graph $G$ has $\Sv(G) = 2$.

As evidence for the Additive Colouring Conjecture, note the following bound on $\mv(G)$:

\begin{prop}[Czerwi\'nski, Grytczuk, \.Zelazny \cite{CGZ}; Chartrand, Okamoto, Salehi, Zhang \cite{COSZ}]\label{multilucky}
For any graph $G$, $\mv(G) \leq \chi(G)$.
\end{prop}

\begin{proof}
Let $k = \chi(G)$ and let $c: V(G) \rightarrow [k]$ be a vertex weighting which is a proper vertex colouring.  Let $m(u)$ denote the multiset of weights from the vertices in $N_G(u)$.  If $uv \in E(G)$, then $c(u) \in m(v) \setminus m(u)$.  Thus, $m(v) \neq m(u)$ for any $uv \in E(G)$ and so $c$ is a vertex weighting which properly colours $V(G)$ by sums.
\end{proof}

Note that the proof above actually shows the following:
\begin{prop}
For any graph $G$, $\sv(G) \leq \chi(G)$.
\end{prop}

Furthermore, it is shown in \cite{COZ} that, for any graph $G$, there exists a set $S_G$ of $\chi(G)$ integers such that there exists a vertex weighting $w: V(G) \rightarrow S_G$ which is an additive colouring.  This result is further extended in \cite{COZ2}, where it is shown that the size of the smallest set of reals such that $G$ has an additive colouring is exactly $\mv(G)$ for any connected graph $G$.

It is easily shown that no constant bound is possible for $\Sv(G)$ by noting that $\Sv(K_n) = n$ for any $n \geq 2$.  In \cite{BBC}, it is shown that $\Sv(G) \leq 468$ if $G$ is planar, $\Sv(G) \leq 36$ if $G$ is planar and $3$-colourable, and $\Sv(G) \leq 4$ if $G$ is planar and has girth at least $13$.  
Czerwi\'nski et al.~\cite{CGZ} prove a general upper bound on $\Sv(G)$ in terms of the acyclic chromatic number of $G$, denoted $A(G)$.

\begin{thm}[Czerwi\'nski, Grytczuk, \.Zelazny \cite{CGZ}]\label{luckybound}
For every graph $G$, $\Sv(G) \leq p_1\cdots p_r$ where $p_i$ denotes the $i^{\textup{th}}$ odd prime number and $r = {A(G) \choose 2}$.
\end{thm}

Unfortunately, the bound given by Theorem \ref{luckybound} can grow large with $\Delta(G)$.  For instance, it is shown in \cite{AMR} that there exist graphs with maximum degree $\Delta(G) = \Delta$ for which $A(G) \in \Omega\left(\frac{\Delta^{4/3}}{(\log\Delta)^{1/3}}\right)$, and hence the bound in Theorem \ref{luckybound} is the product of $r \gg \Omega(\Delta^2)$ primes.  Much improved bounds in the more general setting of list-weightings are presented in Section \ref{additivelistsection}; in particular, we will see that $\Sv(G)  \in O(\Delta(G)^2)$ for every graph $G$.





As with the previously considered variations on the 1-2-3 Conjecture, one may vary the operation used to obtain a vertex colouring from a vertex weighting.  In \cite{COSZ}, Chartrand, Okamoto, Salehi, and Zhang study vertex weightings which properly colour vertices by multisets.  Aside from establishing the easy upper bound of $\chi(G)$ for $\mv(G)$, exact values of $\mv(G)$ are determined for bipartite graphs, complete multipartite graphs, and powers of cycles.

Colouring by sets are considered in \cite{CORZ} by Chartrand, Okamoto, Rasmussen, and Zhang, who show that 
$\sv(G)$ is bounded below by $\left\lceil \log_2(\chi(G)+1) \right\rceil$ and by $1 + \left\lceil \log_2 \omega(G) \right\rceil$.  Particular values are determined for some graph classes, and the effects of vertex and edge deletions are considered.  This parameter is closely related to a graph's {\bf locally identifying chromatic number}, which is the fewest number of colours needed to properly colour $V(G)$ such that the set of colours in $N_G[u]$ differs from the set of colours in $N_G[v]$ for any adjacent vertices $u, v$ with $N_G[u] \neq N_G[v]$.  This concept was introduced in \cite{Esp+} by Esperet et al., who show that there exist graphs which require $\Omega(\Delta(G)^2)$ colours.  They ask whether or not $O(\Delta(G)^2)$ colours suffices for every admissible graph; this was answered in the affirmative by Foucould et al. in \cite{Fou+}.

\section{Variation III:  List weightings}\label{ch:intro:var3}

A popular variation of many colouring problems is to colour elements from independently assigned lists rather than from one universal colour set.  Research on list-colouring problems often provides great insight into classical colouring problems.  For instance, Thomassen's proof that $\ch(G) \leq 5$ for any planar graph $G$ \cite{CTho} is arguably the most elegant proof of the classical result that every planar graph $G$ has $\chi(G) \leq 5$.  
Similarly, we hope to gain insight on the motivating conjectures presented thus far -- the 1-2-3 Conjecture, the 1-2 Conjecture, and the Additive Colouring Conjecture -- by considering colourings that come from list-weightings. 

Let $G$ be a graph and $k, r, s \in \Z^{+}$.  
Assign to each edge $e \in E(G)$ a list of weights $L_e$ and to each vertex $v$ a list of weights $L_v$.  Let $\mathcal{E} = \cup_{e \in E(G)} L_e$, $\mathcal{V} = \cup_{v \in V(G)} L_v$, and $\mathcal{L} = \mathcal{E} \cup \mathcal{V}$.
An {\bf edge list-weighting} of $G$ is a function $w: E(G) \rightarrow \mathcal{E}$ such that $w(e) \in L_e$ for each $e \in E(G)$, a {\bf vertex list-weighting} of $G$ is a function $w: V(G) \rightarrow \mathcal{V}$ such that $w(v) \in L_v$ for each $v \in V(G)$, and a {\bf total list-weighting} of $G$ is a function $w: E(G) \cup V(G) \rightarrow \mathcal{L}$ such that $w(e) \in L_e$ for each $e \in E(G)$ and $w(v) \in L_v$ for each $v \in V(G)$.  If the size of each list is at most $k$, then each weighting is referred to as an {\bf edge $k$-list-weighting}, {\bf vertex $k$-list-weighting}, or {\bf total $k$-list-weighting}, respectively.  If a total list-weighting has $|L_v| \leq r$ and $|L_e| \leq s$ for each $L_v$ and $L_e$, then the weighting is called a {\bf total $(r,s)$-list-weighting}.

Let $\chSe(G)$ denote the smallest value of $k$ for which assigning a list of size $k$ of permissible weights, called a {\bf $k$-list assignment}, to each edge of a graph permits an edge $k$-list-weighting which is a {\bf proper vertex colouring by sums};  each of the parameters given thus far generalizes similarly.

Note that we refer to lists of ``permissible'' weights begin assigned to edges, vertices, or both in the preceding paragraph.  For most colouring operations, these lists may be chosen freely from $\R$.  However, if one were to colour vertices by products from an edge weighting, it is clear that one would never choose an edge weight to be $0$ as this would colour each of its ends $0$.  Hence, we exclude the edge weight $0$ when colouring by products from edge list-weightings (for total and vertex list-weightings, $0$ {\em is} permissible for vertex lists).

\subsection{Edge list-weightings}\label{edgelist}

We will be primarily motivated by the following conjecture, a strengthening of the 1-2-3 Conjecture:

\begin{list123}[Bartnicki, Grytczuk, Niwcyk \cite{BGN09}]
For every nice graph $G$, $\chSe(G) \leq 3$.
\end{list123}

We first note that it is not true that $\chSe(G) = \Se(G)$ for every nice graph $G$.  For example, let $G$ be a path on $4k+2$ vertices for some positive integer $k$ and assign to each edge the list $\{0,a\}$ for some arbitrary $a \in \R^+$; one can easily check that no edge list-weighting exists which properly colours $V(G)$ by sums.

The authors of the List 1-2-3 Conjecture have developed an approach to the problem which makes use of Alon's Combinatorial Nullstellensatz \cite{A99}.  Let $G = (V,E)$ be a graph, with $E(G) = \{e_1, \ldots, e_m\}$.  Associate with each $e_i$ the variable $x_i$, and let $X_{v_j} = \sum_{e_i \ni v_j} \, x_i$.  For an an orientation $D$ of $G$, define the following polynomial:
	\begin{eqnarray*}
	P_D(x_1, \ldots, x_m) = \prod_{(u,v) \in E(D)}(X_v - X_u) \\
	\end{eqnarray*}
Let $w$ be an edge weighting of $G$.  By letting $x_i = w(e_i)$ for $1 \leq i \leq m$, $w$ is a proper vertex colouring by sums if and only if $P_D(w(e_1), \ldots, w(e_m)) \neq 0$.  We can analyze the polynomial $P_D$ by using the Combinatorial Nullstellsatz:

\begin{CN}[Alon \cite{A99}]
Let $\F$ be an arbitrary field, 
and let $f=f(x_1,\ldots,x_n)$ be a polynomial in $\F[x_1,\ldots,x_n]$.
Suppose the total degree of $f$ is $\sum_{i=1}^{n}t_i$, 
where each $t_i$ is a~nonnegative integer, 
and suppose the coefficient of $\prod_{i=1}^{n}x_i^{t_i}$ in $f$ is nonzero.
If $S_1,\ldots,S_n$ are subsets of $\F$ with $|S_i|>t_i$, then
there are $s_1\in S_1, s_2\in S_2,\ldots,s_n\in S_n$ so that
\[
 f(s_1,\ldots,s_n)\neq 0.
\]
\end{CN}

It can be shown that a term in $P_D$ has nonzero coefficient if and only if a related matrix has nonzero permanent (see \cite{BGN09, PW11, Ben-CN} for details).  The following theorem, which is proven using this approach, gives the best known bound on $\chSe(G)$:

\begin{thm}[Seamone \cite{mythesis,Ben-CN}]\label{chSebound}
If $G$ is a nice graph, then $\chSe(G) \leq 2\Delta(G)+1$.
\end{thm}

The Combinatorial Nullstellensatz is also used in \cite{BGN09} to show that $\chSe(G) \leq 3$ if $G$ is complete, complete bipartite, or a tree.  Upper bounds which do not rely on $\Delta(G)$ are proven in \cite{mythesis,Ben-CN} for classes of Cartesian products, though these upper bounds are not constant.

While no improved bounds are known for $\chme(G)$, the Lov\'asz Local Lemma can be used to prove the following results for edge-list-weightings which properly colour vertices by sequences.

\begin{thm}[Seamone, Stevens \cite{Ben-LLL}]\label{listsequencebounds}
If $G$ is a nice graph, then
\begin{enumerate}
\item $\chwe(G) \leq 2$, and
\item $\chWe(G) \leq 3$ if $\delta(G) > \log_3(2\Delta(G)^2 - 2\Delta(G) + 1) + 2$.
\end{enumerate}
\end{thm}

The best known bound on products is a direct corollary of Theorem \ref{chSebound}; one takes a fixed logarithm of the absolute value of every element in each list, and applies the theorem.

\begin{cor}[Seamone \cite{mythesis,Ben-CN}]\label{chPebound}
If $G$ is a nice graph, then $\chSe(G) \leq 4\Delta(G)+2$.
\end{cor}

One may easily extend the concept of vertex colouring edge weightings of graphs to the realm of digraphs.  For an arc-weighting of a digraph $D$, the natural way to define a vertex colouring by sums is to define the colour of a vertex as the sum of the weights of its incident incoming arcs less the sum of the weights of its incident outgoing arcs.  Such vertex colouring edge weightings are completely understood.

\begin{thm}[Bartnicki, Grytczuk, Niwcyk \cite{BGN09}, Khatirinejad, Naserasr, Newman, Seamone, Stevens \cite{Ben2}]\label{BGNdigraph}
For any digraph $D$,  $\chSe(D) \leq 2$.
\end{thm}

The first set of authors cited in Theorem \ref{BGNdigraph} proved this result by a constructive method; the second set used an adaptation of the Combinatorial Nullstellensatz method described above.


\subsection{Total list-weightings}

Just as the 1-2-3 Conjecture has a natural list generalization, so too does the 1-2 Conjecture.

\begin{list12}[Przyby{\l}o, Wozniak \cite{PW11}]
For every graph $G$, $\chSt(G) \leq 2$.
\end{list12}

In \cite{WZ}, Wong and Zhu study {\bf $(k,l)$-total list-assignments}, which are assignments of lists of size $k$ to the vertices of a graph and lists of size $l$ to the edges.  If any $(k,l)$-total list-assignment of $G$ permits a total weighting which is a proper vertex colouring by sums, then $G$ is {\bf $(k,l)$-weight choosable}.  Obviously, if a graph $G$ is $(k,l)$-weight choosable, then $\chSt(G) \leq \max\{k,l\}$.  Wong and Zhu make the following two conjectures which, if true, would be stronger than the List 1-2-3 and List 1-2 Conjectures:

\begin{conj}[Wong, Zhu \cite{WZ}]
Every graph is $(2,2)$-weight choosable.  Every nice graph is $(1,3)$-weight choosable.
\end{conj}

It is easy to prove that every graph $G$ is $(\ch(G),1)$-weight choosable, and hence we have the ``trivial'' upper bound of $\chSt(G) \leq \ch(G) \leq \Delta(G)+1$.  The polynomial method outlined in Section \ref{edgelist} may be adapted to the problem of total weightings (see \cite{PW11} for details).  This approach is used to prove the following improvement on the trivial upper bound on $\chSt(G)$:

\begin{thm}[Seamone \cite{mythesis,Ben-CN}]\label{chStbound}
Any graph $G$ is $(\lceil\frac{2}{3}\Delta(G)\rceil + 1, \lceil\frac{2}{3}\Delta(G)\rceil + 1)$-weight choosable, and hence $\chSt(G) \leq \lceil\frac{2}{3}\Delta(G)\rceil + 1$.
\end{thm}

The other known results on $(k,l)$-weight choosability focus on graph classes which are $(k,l)$-weight choosable for small values of $k$ and $l$ (see \cite{BGN09}, \cite{PY}, \cite{PW11}, \cite{WZ}, \cite{WWZ}, \cite{WYZ}).  Table \ref{wctable} on page \pageref{wctable} summarizes these results, most of which are also proven using the Combinatorial Nullstellensatz approach described in Section \ref{edgelist}.

\begin{table}[h!]
\begin{center}
	\begin{tabular}{| c | c | }
    	\hline
    	Type of graph & $(k,l)$-weight choosability  \\ \hline
	$K_2$ & $(2,1)$  \\
	$K_n$, $n \geq 3$ & $(1,3)$, $(2,2)$   \\
	$K_{n,m}$, $n \geq 2$ & $(1,2)$ \\
	$K_{m,n,1,\ldots,1}$ & $(2,2)$   \\
	trees & $(1,3)$, $(2,2)$   \\
	unicyclic graphs & $(2,2)$  \\
	generalized theta graphs & $(1,3)$, $(2,2)$  \\
	wheels & $(2,2)$, $(1,3)$  \\
	hypercubes $Q_d$, $d$ even or $d=3$ & $(1,3)$ \\
	planar & $(1,10)$ \\
	outerplanar & $(1,4)$ \\
	$s$-degenerate graphs $s \geq 2$ & $(1,2s)$ \\
	\hline
  	\end{tabular}
  \caption{$(k,l)$-weight choosability of classes of graphs}
  \label{wctable}
\end{center}
\end{table}

Hor\u{n}\'ak and Wo\'zniak \cite{HW} consider the list variation of $\se(G)$ (or $\textup{gndi}(G)$).  They denote this parameter $\textup{glndi}(G)$, while we denote it $\chse(G)$.  They show that $\chse(G) = \se(G)  = \lceil{\log_2{\chi(G)}}\rceil + 1$ in the case where $G$ is a path or cycle, and that $\chse(T) \leq 3$ for any tree $T$.  If the edge weighting is required to give a proper edge colouring as well, they denote the corresponding parameter $\textup{lndi}(G)$ in correspondence with the parameter $\textup{ndi}(G)$ for the non-list version of the problem.  They show that $\textup{lndi}(G) = \textup{ndi}(G)$ if $G$ is a cycle or a tree and conjecture that $\textup{lndi}(G) = \textup{ndi}(G)$ for every graph $G$, a conjecture reminiscent of the List Colouring Conjecture (see \cite{JT} for a history of the List Colouring Conjecture).

We have already noted that $\chSt(G) \leq \ch(G)$ for any graph $G$.  This upper bound also holds for colourings by multisets and by products.  It follows that the following graphs, which all have low choosability number, have small values of $\chSt(G)$, $\chPt(G)$ and $\chmt(G)$ (references stated are for the bound on $\ch(G)$ used to obtain the result on weighting parameters):

\begin{prop}  Let $n$ be a positive integer.
\begin{enumerate}
\item \textup{(Erd{\H o}s, Rubin, Taylor \cite{ERT})} If $G$ is a graph whose core\footnote{The core of a graph is the subgraph obtained by repeated deletion of vertices of degree 1 until none remain.  This is one of two standard definitions for the core of a graph, and should not be confused with the definition from the study of graph homomorphisms.} is $K_1$, $C_{2n+2}$, or $\theta_{2,2,2n}$, then $\chSt(G), \chPt(G),$ and $\chmt(G)$ are all at most $2$.
\item \textup{(Akbari, Ghanbari, Jahanbekam, Jamaali \cite{AGJJ})} If all cycles of $G$ have length divisible by an integer $k \geq 3$, then $\chSt(G), \chPt(G),$ and $\chmt(G)$ are all at most $3$.
\item \textup{(Thomassen \cite{CTho})} If $G$ is planar, then $\chSt(G), \chPt(G),$ and $\chmt(G)$ are all at most $5$.
\item \textup{(Alon, Tarsi \cite{AT92})} If $G$ is bipartite and planar, then $\chSt(G), \chPt(G),$ and $\chmt(G)$ are all at most $3$.
\end{enumerate}
\end{prop}

Furthermore, every graph $G$ has the following two upper bounds on $\chSt(G)$, $\chPt(G)$, and $\chmt(G)$:

\begin{prop}\label{chbound}
If $G$ is a graph on $n$ vertices, then
	\begin{enumerate}
	\item \textup{(Erd{\H o}s, Rubin, Taylor \cite{ERT})} each of $\chPt(G)$ and $\chmt(G)$ is at most $\Delta(G) + 1$, and
	\item \textup{(Eaton \cite{Eaton})} each of $\chSt(G), \chPt(G),$ and $\chmt(G)\}$ is at most $\chi(G)\log{n}$.
	\end{enumerate}
\end{prop}

Note that the bound $\chSt(G) \leq \Delta(G) + 1$ is omitted since a stronger result is given in Theorem \ref{chStbound}.

We conclude with a brief consideration of colouring vertices by products from list-weightings.  By excluding $0$ from our ground set of potential edge weights, it remains possible that $\chPe(G) \leq 3$ (were we not to exclude $0$, the lists $\{-1,0,1\}$ would not permit a vertex colouring for a nice non-bipartite graph).  However, by assigning every vertex and edge the list $\{1,-1\}$, one sees that graphs which are not bipartite cannot be totally weighted from these lists in such a way as to properly colour vertices by products.  So, while the List 1-2 Conjecture posits that $\chSt(G) \leq 2$ for any graph $G$, a similar bound for total list-weightings is not possible for all graphs when colouring by products.

\subsection{Vertex list-weightings}\label{additivelistsection}

Denote by $\chSv(G)$ the smallest $k$ such that $G$ has an additive colouring from any assignment of lists of size $k$ to the vertices of $G$, and call $\chSv(G)$ the {\bf additive choosability number} of $G$.  Such list-weightings were considered by Czerwi\'nski et al. \cite{CGZ}, who show that 
if $T$ is a tree, then $\chSv(T) \leq 2$, and if $G$ is a bipartite planar graph, then $\chSv(G) \leq 3$.  

The following conjecture follows in the spirit of the List 1-2-3 Conjecture:

\begin{AClistconj}[Seamone \cite{mythesis, Ben-CN}]
For any graph $G$, $\chSv(G) = \Sv(G)$.
\end{AClistconj}

While there is no known bound on $\chSv(G)$ in terms of $\Sv(G)$ (a bound in terms of $\ch(G)$ or $\chi(G)$ would also be of interest), one may bound $\chSv(G)$ by a function of the maximum degree of $G$:

\begin{thm}[Akbari, Ghanbari, Manaviyat, Zare \cite{AGMZ}]\label{listluckybound}
For any graph $G$, $\chSv(G) \leq \Delta(G)^2 - \Delta(G) + 1$.
\end{thm}

The following general bound can be proven with a straightforward inductive argument:

\begin{thm}\label{additivechoosability}
If $G$ is a $d$-degenerate graph, then $\chSv(G) \leq d\Delta(G) + 1$.
\end{thm}

It the easily follows that $\chSv(G) \leq \Delta(G)\left(\chi(G) - 1\right) + 1$ if $G$ is chordal, $\chSv(G) \leq 2\Delta+1$ if $G$ is $K_4$-minor free, and $\chSv(G) \leq 5\Delta(G)+1$ if $G$ is planar.
It is also shown in \cite{mythesis,Ben-CN} that $\chSv(G) \leq \chi(G)$ when $G$ is a complete multipartite graph (in fact, a slightly stronger statement is proven, one which generalizes a result of Chartrand, Okamoto, and Zhang \cite{COZ} who considered $\Sv(G)$ for regular, complete multipartite graphs).

\section{Variation IV:  Edge colourings}\label{ch:intro:var4}

Each section so far has dealt with the problem of properly colouring the {\em vertices} of a graph given an edge weighting, vertex weighting, or total weighting.  We now define how one might similarly define edge colourings and total colourings from a weighting.

Given an edge $k$-weighting, we colour an edge $e \in E(G)$ with the sum, product, multiset, or set of weights from edges adjacent to $e$.  If the resulting edge colouring is proper, then the $k$-edge weighting of $G$ is a {\bf proper edge colouring by sums, products, multisets or sets}, respectively.  The smallest $k$ for which such colourings exist for a graph $G$ are denoted $\eSe(G)$, $\ePe(G)$, $\eme(G)$ and $\ese(G)$, respectively.

These parameters do not appear in the present literature, but a few bounds are easily obtained by noting that each is a particular case of its vertex-weighting vertex-colouring analog, e.g., $\eSe(G) = \Sv(L(G))$, $\echSe(G) = \chSv(L(G))$, etc.  In the particular case of $\eSe(G)$, we have a few easy to prove upper bounds.  Theorem \ref{luckybound} bounds $\eSe(G)$ in terms of the acyclic chromatic number of $L(G)$.  The case of proper colouring by multisets also has a bound which follows from previous work.

\begin{prop}
For any graph $G$, $\eme(G) \leq \chi'(G)$.
\end{prop}

\begin{proof}
By Proposition \ref{multilucky}, $\mv(L(G)) \leq \chi(L(G))$.  Since $\eme(G) = \mv(L(G))$, the result follows.
\end{proof}

One may also colour edges from a vertex or total weighting, or by list weightings of each type; such scenarios have yet to be studied.



A total $k$-weighting of a graph $G$, $w: V(G) \cup E(G) \rightarrow [k]$ is a {\bf proper edge colouring by sums} if the edge colouring function $c(uv) =  w(u) + w(v) + w(uv)$ is proper. 
The smallest value of $k$ such that there exists a total $k$-weighting which is a proper edge colouring by sums is denoted $\eSt(G)$.
While one may easily define similar parameters for proper colourings by products, multisets, or sets, only $\eSt(G)$ has been explicitly studied in published literature.  

\begin{prop}[Brandt, Budajov\'a, Rautenbach, Stiebitz \cite{BBRS}]
If a graph $G$ has maximum degree $\Delta(G) = \Delta$, then $\eSt(G) \leq \frac{\Delta}{2} + O(\sqrt{\Delta\log{\Delta}})$. 
\end{prop}

Brandt et al. also establish bounds for 
trees and complete graphs, and give necessary and sufficient conditions for a cubic graph $G$ to have $\eSt(G) = 2$.

The topic of weightings and list-weightings which are proper total colourings via binary operations remains unstudied.

\section{Distinguishing colourings and irregularity strength}\label{global}

As mentioned, one motivation for the 1-2-3 Conjecture is the study of graph irregularity strength, which is is the smallest positive integer $k$ for which there exists an edge weighting $w: E(G) \to [k]$ such that $\sum_{e \ni u} w(e) \neq \sum_{e \ni v} w(e)$ for all $u,v \in V(G)$.  
Irregularity strength and its variations have generated a significant amount of interest since the topic's inception in 1986.  We will only highlight a few of the more significant results in the field here; the reader is referred to \cite{survey}, \cite{Leh91}, and \cite{West_IS} for a more complete survey of results on graph irregularity strength.

The original notation for the irregularity strength of a graph $G$, given by Chartrand et al. \cite{CJL+}, is $\textup{s}(G)$; to maintain consistency with the notation developed so far, we will denote it $\textup{s}_{\Sigma}^e(G)$.  There is a natural ``global" variation of each of the ``local'' parameters defined so far.  For the local parameters $\chi_x^y(G)$, ${\chi'}_x^y(G)$, and ${\chi''}_x^y(G)$, where $x \in \left\{\sum, \prod, m, s, \sigma, \sigma^* \right\}$ and $y \in \left\{e,v,t\right\}$, let $\textup{s}_x^y(G)$, ${\textup{s}'}_x^y(G)$, and ${\textup{s}''}_x^y(G)$ denote their respective ``global" counterparts.  Similarly, let $\textup{ls}_x^y(G)$, ${\textup{ls}'}_x^y(G)$, and ${\textup{ls}''}_x^y(G)$ denote the global versions of $\ch_x^y(G)$, ${\ch'}_x^y(G)$, and ${\ch''}_x^y(G)$, respectively.

A major conjecture on the topic of irregularity strength was one of Aigner and Triesch \cite{AT90}, who asked if $\textup{s}_{\Sigma}^e(G) \leq |V(G)|-1$ for every nice graph.  This conjecture was answered in the affirmative by Nierhoff \cite{Nie}.  The following is the best known irregularity strength bound:

\begin{thm}[Kalkowski, Karo\'nski, Pfender \cite{KKP3}]
For any nice graph $G$, $\textup{s}_{\Sigma}^e(G) \leq \lceil 6n/\delta(G) \rceil$.
\end{thm}


Recently, Anholcer and Cichacz proved a group analog of Aigner and Triesch's conjecture.  They define the {\bf group irregularity strength} of a graph $G$, $s_g(G)$, to be the smallest value of $s$ such that, for any abelian group $(\Gamma,+)$ of order $s$, there exists an edge weighting of $G$ from $\Gamma$ such that distinct vertices receive distinct sums of edge weights.

\begin{thm}[Anholcer, Cichacz \cite{AC2}]
For any connected graph of order $n \geq 3$, $n \leq s_g(G)  \leq n+2$.
\end{thm}

As with their result on group sum chromatic numbers (Theorem \ref{groupsum}), Anholcer and Cichacz characterize which graphs have group irregularity strength $n$, $n+1$, and $n+2$.
  
Note that irregular weightings are closely related to antimagic labelings.  A graph $G$ with $m$ edges is called {\bf antimagic} if there exists a bijection between $E(G)$ and $\{1, \ldots, m\}$ such that every vertex has a distinct sum of edge weights from its adjacent edges.  Such a bijection is called an {\bf antimagic labeling}.  Antimagic labelings and their variations are widely studied; a thorough survey may be found in Chapter 6 of \cite{survey}.

Aigner et al. \cite{ATT} introduce a version of irregularity strength where each vertex is assigned the multiset of edge weights from incident edges rather than the sum; the corresponding parameter is $\textup{s}_{m}^e(G)$ in our notation and is called the {\bf multiset irregularity strength} of $G$.  The following result is proven using probabilistic methods:

\begin{thm}[Aigner, Triesch, Tuza \cite{ATT}]\label{ATT}
If $G$ is a $d$-regular graph, then $\textup{ls}_{m}^e(G) \in \Theta(n^{1/d})$.
\end{thm} 

Note that the original result published in \cite{ATT} does not refer explicitly to weighting from lists, however their proofs all directly translate to this setting.  The constant in the upper bound suppressed by the $\Theta$ notation can be improved if one relaxes the colouring method to sequences given any ordering of the edges \cite{Ben-LLL} ($\textup{ls}_{\sigma}^e(G)$ is the {\bf general sequence irregularity strength} of $G$; see Section \ref{ch:intro:var1} for more on colouring by sequences).  More generally, it is shown in \cite{Ben-LLL} that if $G$ is a graph on $n$ vertices with minimum degree $\delta(G) > c\log{n}$ for large enough $c=c(k)$, then $\textup{ls}_\sigma^e(G) \leq k$.  Similar results are also given for total list-weightings.

Anholcer \cite{Anh} studied the {\bf product irregularity strength} of a graph, or $\textup{s}_{\Pi}^e(G)$.  He establishes bounds on $\textup{s}_{\Pi}^e(G)$ when $G$ is a cycle, path or grid.  Pikhurko \cite{Pik} has shown that if a graph is sufficiently large, then $\textup{s}_{\Pi}^e(G) \leq |E(G)|$ (in antimagic labelling terminology, he shows that sufficiently large graphs are product antimagic).

A {\bf vertex-distinguishing edge-colouring} of a graph $G$ is a colouring of the edges such that for any two vertices $u$ and $v$, the set of colours assigned to the set of edges incident to $u$ differs from the set of colours assigned to the set of edges incident to $v$.  Harary and Plantholt \cite{HP85} introduced the study of vertex-distinguishing {\em proper} edge-colourings in 1985; they use $\chi_0(G)$ to denote the smallest value of $k$ such that a graph $G$ has such an edge-colouring, and they call the parameter the {\bf point-distinguishing chromatic index} of $G$.  Note that $\chi_0(G)$ is a global version of $\se(G)$ with the added constraint that the edge weighting is a proper edge-colouring, and hence is an upper bound on $\textup{s}_s^e(G)$.   

\begin{thm}[Burris, Schelp \cite{BS97}]\label{setirrstrength}
If $n_i$ denotes the number of vertices of degree $i$ in a graph $G$, then for any nice graph $G$, $\chi_0(G) \leq C\max\{n_i^{1/i} \,:\, 1 \leq i \leq \Delta(G)\}$, where $C$ is a constant relying only on $\Delta(G)$.
\end{thm}

Burris and Schelp also make the following two conjectures:

\begin{conj}[Burris, Schelp \cite{BS97}]
\begin{enumerate}
\item Let $G$ be a simple graph with no more than one isolated vertex and no connected component isomorphic to $K_2$.  If $k$ is the minimum integer such that ${k \choose d} \geq n_d$ for all $\delta(G) \leq d \leq \Delta(G)$, then $\chi_0(G) \in \{k, k+1\}$.
\item If $G$ is a simple graph with no more than one isolated vertex and no connected component isomorphic to $K_2$, then $\chi_0(G) \leq |V(G)| + 1$.
\end{enumerate}
\end{conj}

The second conjecture was answered in the affirmative by Bazgan et al. \cite{Baz+}.



Ba\v{c}a, Jendrol', Miller, and Ryan introduced the {\bf total edge irregularity strength} of a graph $G$ in \cite{BJMR}, which is defined to be the smallest $k$ for which there is a total $k$-weighting $w$ such that $w(u) + w(v) + w(uv) \neq w(x) + w(y) + w(xy)$ for any distinct $uv, xy \in E(G)$.  In our notation, this parameter is ${\textup{s}'}_\Sigma^t(G)$.  They show that if $G$ has $m$ edges, then $\left\lceil\frac{m+2}{3}\right\rceil \leq {\textup{s}'}_\Sigma^t(G) \leq m$.  They also make the following conjecture:

\begin{conj}[Ba\v{c}a, Jendrol', Miller, Ryan \cite{BJMR}]
If $G \ncong K_5$, then $${\textup{s}'}_\Sigma^t(G) = \max\left\{ \left\lceil \dfrac{\Delta(G)+1}{2} \right\rceil, \left\lceil\dfrac{|E(G)|+2}{3}\right\rceil \right\}.$$
\end{conj}

This conjecture was recently confirmed in \cite{BMR} for graphs with $|E(G)| > 111\,000\Delta(G)$ and for a random graph from $\mathcal{G}(n,p)$ for any $p = p(n)$.  It is also shown in \cite{BMR} that ${\textup{s}'}_\Sigma^t(G) \leq \left\lceil \frac{|E(G)|}{2}\right\rceil$ for any graph $G$ which is not a star.  The conjecture is also verified in \cite{BMRRR} for any graph $G$ having $|E(G)| \leq \frac{3}{2}|V(G)| - 1$ and $\Delta(G) \leq \left\lceil\frac{|E(G)|+2}{3}\right\rceil$.

Ba\v{c}a et al. \cite{BJMR} also introduce the {\bf total vertex irregularity strength} of a graph $G$, denoted ${\textup{s}}_\Sigma^t(G)$, which is the smallest $k$ such that $G$ has a total $k$-weighting $w$ such that $w(u) + \sum_{e \ni u}w(e) \neq w(v) + \sum_{e \ni v} w(e)$ for any distinct $u,v \in V(G)$.  The best known bound on ${\textup{s}}_\Sigma^t(G)$ was given by Anholcer, Kalkowski, and Przyby{\l}o \cite{AKP}, who showed that ${\textup{s}}_\Sigma^t(G) \leq 3\left\lceil n/\delta \right\rceil + 1$ for a graph on $n$ vertices with minimum degree $\delta(G) = \delta$.  For a survey of results on total edge irregularity strength and total vertex irregularity strength, see Section 7.17 in \cite{survey}.

Chartrand, Lesniak, VanderJagt, and Zhang \cite{CLVZ} have very recently considered vertex weightings for which the multiset of colours of the vertices in $N(u)$ differs from the multiset of colours of the vertices of $N(v)$ for any distinct vertices $u$ and $v$.  The smallest $k$ such that a graph $G$ has such a vertex $k$-weighting is denoted ${\textup{s}}_m^v(G)$.  Such weightings are also known as {\bf recognizable colourings} and ${\textup{s}}_m^v(G)$ is the {\bf recognition number} of $G$.  Recognition numbers are determined in \cite{CLVZ} for a number of classes of graphs.

An {\bf edge-distinguishing vertex weighting} of a graph $G$ is a weighting $w:V(G) \to [k]$ such that $\{w(u),w(v)\} \neq \{w(x),w(y)\}$ for any two distinct edges $uv, xy \in E(G)$.  Such weightings were introduced by Frank, Harary and Plantholt \cite{FHP}, where they were called {\bf line-distinguishing vertex colourings}.  The smallest value of $k$ for which a graph $G$ has an edge-distinguishing vertex-weighting is called the {\bf line-distinguishing chromatic number}, and is denoted ${\textup{s}'}_{s}^v(G)$ (note that $\{w(u),w(v)\} \neq \{w(x),w(y)\}$ as sets if and only if they are distinct as multisets, and hence we could also use ${\textup{s}'}_{m}^v(G)$).  Harary and Plantholt \cite{HP83} conjectured that ${\textup{s}'}_{m}^v(G) \geq \chi'(G)$ for every graph $G$, which was proven by Salvi \cite{Salvi}.

%

\section{Summary}

Although the ultimate goal of showing that $\Se(G) \leq 3$ remains out of reach for the moment, a good deal of progress has been made in many of the natural generalizations.  A summary of the major results contained in the thesis are found in Table \ref{summaryofbounds}.  In each case, assume that $G$ is a nice graph if necessary, and that $D$ is a digraph.  Not every bound is explicitly stated; for instance, since $\We(G) \leq \me(G)$ and no bound for $\We(G)$ is known other than the best known bound on $\me(G)$, $\We(G)$ is omitted from the table. \\

\begin{table}[htb]
\begin{center}
	\begin{tabular}{| c | c | c | c |}
    	\hline
    	Parameter & Conjectured upper bound & Known upper bound & Reference   \\ \hline
	$\Se(G)$ & 3 & 5 & \cite{KKP1} \\
	$\me(G)$ & 3 & 4 & \cite{AADR05} \\
	$\Pe(G)$ & 3 & 5 & \cite{SK08} \\
	$\se(G)$ & -- & $\left\lceil \log_2\chi(G)\right\rceil + 1 $ & \cite{GP09} \\
	\hline
	$\St(G)$ & 2 & 3 & \cite{Kal} \\
	$\Pt(G)$ & 2 & 3 & \cite{SK08} \\
	\hline
	$\chSe(G)$ & 3 & $2\Delta(G)+1$ &  \cite{mythesis, Ben-CN} \\
	$\chSt(G)$ & 2 & $\left\lceil \frac{2}{3}\Delta(G) \right\rceil + 1$ & \cite{mythesis, Ben-CN} \\
	$\chPe(G)$ & -- & $4\Delta(G)+2$ & \cite{mythesis, Ben-CN} \\
	$\chwe(G)$ & 2 & 2 & \cite{Ben-LLL} \\
	\hline
	$\chSe(D)$ & 2 & 2 & \cite{BGN09}, \cite{Ben2} \\
	\hline
	$\chSv(G)$ & $\Sv(G)$, $\chi(G)$ & $\Delta(G)^2 + 1$ & \cite{mythesis, Ben-CN} \\
	\hline
	$\eSt(G)$ & -- & $\frac{\Delta(G)}{2} + O(\sqrt{\Delta(G)\log{\Delta(G)}})$ & \cite{BBRS} \\
	\hline
	\end{tabular}
  \caption{Summary of derived colouring parameter values}
  \label{summaryofbounds}
\end{center}
\end{table}

\section{Acknowledgements}

Much of this work was compiled during the writing of the author's doctoral thesis.  As such, a great debt is owed to Brett Stevens for his supervisory role, Carleton University for their institutional support, and to the Natural Sciences and Engineering Research Council of Canada and the Fonds de recherche Qu\'ebec - Nature et technologies for their financial support.

\bibliographystyle{plain}
\bibliography{references}

\end{document}